\documentclass[a4paper]{article}%
\usepackage{amsmath}
\usepackage{amsfonts}
\usepackage{amssymb}
\usepackage{graphicx}%
\providecommand{\U}[1]{\protect\rule{.1in}{.1in}}
\newtheorem{theorem}{Theorem}

\newtheorem{corollary}[theorem]{Corollary}

\newtheorem{lemma}[theorem]{Lemma}

\newenvironment{proof}[1][Proof]{\noindent\textbf{#1.} }{\ \rule{0.5em}{0.5em}}

\begin{document}
	
\title{Asymptotics of Titchmarsh--Weyl functions near the real axis and
	an application to the KdV hierarchy}
\author{Shuo Zhang$^{1}$}
\date{}
\maketitle

\begin{abstract}
	We establish high-energy asymptotic expansions of Titchmarsh--Weyl
	functions for
	one-dimensional Schr\"odinger and Dirac operators in regions whose
	boundaries approach the spectrum at a prescribed polynomial
	rate.  For bounded potentials with bounded derivatives, the expansions
	remain uniform up to these boundaries, with an explicit loss in the
	remainder determined by the rate of approach.  We treat both self-adjoint
	Dirac operators and non-self-adjoint Dirac operators with skew-adjoint
	potential matrices, keeping track of the distinct half-planes in which the
	two scalar Weyl coordinates are naturally defined.  As an application of
	the Schr\"odinger expansion, we verify the high-energy hypothesis in
	Kotani's construction of KdV flows: for every odd integer $p\geq3$, each
	real-valued $q\in W^{2p-1,\infty}(\mathbb{R})$ generates a global classical
	solution of the member of the KdV hierarchy indexed by $(p+1)/2$.
\end{abstract}

\section{Introduction}
\footnotetext[1]{School of Math. Nanjing Univ. Nanjing China shuozhang@smail.nju.edu.cn}

A principal motivation for studying Titchmarsh--Weyl asymptotics close to
the spectrum comes from the construction of KdV flows in the
Sato--Segal--Wilson framework.  In Kotani's formulation \cite{Kotani}, the
spectral data are assembled from half-line Titchmarsh--Weyl functions on an
exterior domain whose boundary approaches the spectrum at high energy.
Consequently, the usual asymptotics in fixed nonreal sectors do not by
themselves verify the hypotheses needed for the flow construction.  This is
the reason for seeking uniform expansions in the narrowing domains studied
below.

We briefly recall the scalar Titchmarsh--Weyl function.  Let
\[
H_q=-\frac{d^2}{dx^2}+q(x)
\]
on $\mathbb R_+=[0,\infty)$, and suppose that the expression is limit point
at $+\infty$.  For $z\in\mathbb C\setminus\mathbb R$, let
$\psi_+(\cdot,z)$ be, up to a nonzero constant factor, the solution of
$H_q\psi=z\psi$ that is square integrable near $+\infty$.  The corresponding
Titchmarsh--Weyl function is
\begin{equation}
	m_+^{\rm S}(x,z)=\frac{\psi_+'(x,z)}{\psi_+(x,z)}.
	\label{1.1}
\end{equation}
It encodes the half-line spectral data and is independent of the
normalization of $\psi_+$.  By analytic continuation, we use the same
notation at sufficiently negative real values of $z$, below the spectrum of
a fixed self-adjoint half-line realization; these are the only real spectral
parameters that occur in the high-energy estimates below.

The study of high-energy Weyl--Titchmarsh asymptotics goes back to Weyl's
introduction of the $m$-function in his spectral theory of singular
differential equations \cite{Weyl}.  For scalar half-line Schr\"odinger
operators, Marchenko \cite{Marchenko1952} proved the fundamental leading-order relation
\[
m^{\rm S}(0,z)=-\sqrt{-z}\,(1+o(1)),
\qquad |z|\to\infty,
\]
uniformly in closed nonreal sectors, and Levitan  \cite{Levitan1952} subsequently gave a shorter
proof.  Here the square root is chosen so
that $\operatorname{Re}\sqrt{-z}>0$.  Everitt and Everitt--Halvorsen
\cite{Everitt1972,EverittHalvorsen1978}
developed related asymptotic criteria for the Titchmarsh--Weyl coefficient.  Atkinson \cite{Atkinson1981} then obtained a more local
refinement: for a locally integrable potential, the first correction to the
free term can be expressed through a Laplace-type transform of $q$ near the
finite endpoint.  In particular, at a Lebesgue point of
$q$ this yields
\[
m^{\rm S}(0,z)
=-\sqrt{-z}-\frac{q(0)}{2\sqrt{-z}}+o(|z|^{-1/2}).
\]
Bennewitz \cite{Bennewitz1989} placed such leading spectral asymptotics in a more general
Sturm--Liouville setting and clarified how little coefficient regularity is
needed for them.

For smoother potentials, the emphasis shifted from the leading term to
finite or complete inverse-power expansions.  Harris \cite{HarrisSchrodinger} showed that if
$q\in C^N[0,\delta)$, then
\[
m^{\rm S}(0,z)
\sim i\sqrt z+
\sum_{j\geq1}a_j(0)(i\sqrt z)^{-j},
\]
truncated at the order permitted by $N$, where the coefficients are local
differential polynomials in $q$.  These
results make the locality of the coefficients transparent, but their
uniform remainder estimates are formulated in fixed closed sectors
\[
S_a=\{z\in\mathbb C_+:a\leq\arg z\leq\pi-a\},
\qquad 0<a<\frac{\pi}{2}.
\]
Thus $|\operatorname{Im}z|\geq(\sin a)|z|$, and they do not directly cover
high-energy sequences whose angle with the real axis tends to zero.

A particularly important change of viewpoint was introduced by Barry
Simon.  Instead of recording only a finite list of inverse-power
coefficients, Simon \cite{Simon1999} represented the high-energy correction to the free
Weyl function by an $A$-amplitude: for $\kappa=\sqrt{-z}$ in a right
half-plane and each fixed $a>0$,
\[
m^{\rm S}(0,-\kappa^2)
=-\kappa-\int_0^a A(\alpha)e^{-2\alpha\kappa}\,d\alpha
+\text{an exponentially small remainder}.
\]
This formulation encodes the potential near the endpoint in a Laplace
transform and makes the local character of high-energy data quantitative.  Fritz Gesztesy and Simon
\cite{GesztesySimon2000} extended this formalism from
the initial class of potentials to general real locally integrable
potentials, including limit-circle situations, and established a direct
connection between the $A$-amplitude and the spectral measure.  Their local Borg--Marchenko theorem then showed
that exponentially small high-energy differences of two Weyl functions
determine equality of the corresponding potentials on an initial interval
\cite{GesztesySimonBorg2000}.  In this sense, the Simon--Gesztesy theory
goes beyond a formal asymptotic series: it identifies precisely how local
coefficient information is stored in exponentially accurate spectral
asymptotics.

Gesztesy's subsequent work with Clark 
\cite{ClarkGesztesySchrodinger} carried both aspects of this program
to matrix-valued systems.  For self-adjoint matrix Schr\"odinger operators,
Clark and Gesztesy established sectorial high-energy expansions of the
matrix Weyl--Titchmarsh function, supplied recursive formulas for the
coefficients, and proved associated local uniqueness results.  More recently, the leading asymptotic
theory has also been extended to Schr\"odinger operators with
measure-valued potentials, demonstrating that the first-order Weyl
asymptotics persist far below the smooth setting required for a long
inverse-power expansion \cite{LugerTeschlWoehrer}.

We shall also consider the two Dirac operators
\[
D_q^{\rm sa}
=i\begin{pmatrix}1&0\\0&-1\end{pmatrix}\frac{d}{dx}
+\begin{pmatrix}0&q\\\overline q&0\end{pmatrix}
\]
and
\[
D_q^{\rm ns}
=i\begin{pmatrix}1&0\\0&-1\end{pmatrix}\frac{d}{dx}
+\begin{pmatrix}0&q\\-\overline q&0\end{pmatrix}.
\]
The first operator is self-adjoint under the usual hypotheses.  The second
is non-self-adjoint; although its potential matrix is skew-adjoint, the full
differential operator is in general neither self-adjoint nor
skew-self-adjoint.

The Dirac theory developed in parallel.  Everitt, Hinton, and Shaw  \cite{EverittHintonShaw1983} obtained the leading asymptotic form of the Titchmarsh--Weyl coefficient for Dirac
systems.  Under finite smoothness assumptions
near a regular endpoint, Harris \cite{HarrisDirac} derived a finite high-energy expansion whose
coefficients are determined by endpoint values of the potential and its
derivatives, uniformly in every fixed closed subsector of the upper
half-plane.  Clark and Gesztesy \cite{ClarkGesztesyDirac} later treated general
matrix-valued Dirac-type operators, obtaining high-energy Weyl--Titchmarsh
matrix asymptotics together with local uniqueness, trace, and Borg-type
results.

For both Schr\"odinger and self-adjoint Dirac operators on the half-line with integrable coefficients, Hinton, Klaus, and Shaw \cite{HintonKlausShaw1989} obtained convergent series representations of the associated Titchmarsh--Weyl $m$-functions and, assuming integrability of the derivatives up to order $N$, derived high-energy expansions through order $\lambda^{-N}$ valid in the closed upper half-plane, including along the real axis.

For non-self-adjoint Dirac systems, a
Weyl theory in the appropriate exterior half-planes, including existence,
uniqueness, and inverse reconstruction, was developed for skew-self-adjoint
systems by Fritzsche, Kirstein, Roitberg, and Sakhnovich \cite{FKRS}.

The common geometric limitation in the classical high-energy expansions is
their sectorial character: the spectral parameter must remain separated
from the real axis by a fixed positive angle (or, in the non-self-adjoint
setting, stay in a fixed exterior half-plane).

The purpose of the present paper is to retain these expansions in domains
whose boundaries approach the real axis.  Let $N\geq1$,
$0\leq n\leq N-1$, and choose a positive even function $\gamma_n$ such that,
for some constants $c_n,C_n>0$,
\begin{equation}
	c_n(1+|t|)^{-n}\leq\gamma_n(t)
	\leq C_n(1+|t|)^{-n},
	\qquad t\in\mathbb R.
	\label{1.2}
\end{equation}
For the Schr\"odinger problem on $\mathbb R_+$, define
\begin{equation}
	\Omega_{n,+}^{\rm S}
	=\{k=\lambda+i\eta\in\mathbb{C}:\lambda>\gamma_n(\eta)\}.
	\label{1.3}
\end{equation}
When $z\in\mathbb{C}\setminus[0,\infty)$, $\sqrt{-z}$ denotes the branch
with positive real part.  Thus the natural parameter for
$m_+^{\rm S}$ is $k=\sqrt{-z}\in\Omega_{n,+}^{\rm S}$.

Our first theorem extends Harris's fixed-sector expansion to a domain whose
boundary approaches the imaginary axis in the $k$-plane at a prescribed
polynomial rate.

\begin{theorem}
	\label{ t1}
	Let $q\in W^{N,\infty}(\mathbb{R}_+;\mathbb{R})$ and
	$0\leq n\leq N-1$.  There exist differential polynomials $c_j$ in $q$
	such that, as $|k|\to\infty$,
	\begin{equation*}
		m_+^{\rm S}(0,-k^2)
		=-k+\sum_{j=1}^{N-n-1}c_j(0)k^{-j}+O(|k|^{n-N}),
		\qquad k\in\Omega_{n,+}^{\rm S}.
	\end{equation*}
	The coefficients are determined by
	\[
	c_1=-\frac{q}{2}, \qquad c_{j+1}=\frac12\left(c_j'+\sum_{\ell=1}^{j-1}c_\ell c_{j-\ell}\right), \qquad j\geq1.
	\]
\end{theorem}

For a potential $q$ on the whole line, let $\psi_-(\cdot,z)$ be the solution
square integrable near $-\infty$ and use the standard Herglotz convention
\[
m_-^{\rm S}(0,z;q)=-\frac{\psi_-'(0,z)}{\psi_-(0,z)}.
\]
Put $\check q(x)=q(-x)$.  Reflection in the origin then gives the exact
identity
\begin{equation}
	m_-^{\rm S}(0,z;q)=m_+^{\rm S}(0,z;\check q).
	\label{1.4}
\end{equation}
Thus the left half-line statement need not be included separately in the
theorem.  Introduce, only for the two-sided construction,
\[
\Omega_{n,-}^{\rm S}
=\{k=\lambda+i\eta\in\mathbb C:-\lambda>\gamma_n(\eta)\}.
\]
Since $\gamma_n$ is even, $k\in\Omega_{n,-}^{\rm S}$ implies
$-k\in\Omega_{n,+}^{\rm S}$.  Applying
Theorem~\ref{ t1} to $\check q$ with parameter $-k$
gives the expansion of $m_-^{\rm S}$ for
$k\in\Omega_{n,-}^{\rm S}$.  In the conventions used here,
\begin{equation}
	m_-^{\rm S}(0,-k^2)
	=k-\sum_{j=1}^{N-n-1}c_j(0)k^{-j}+O(|k|^{n-N}),
	\qquad k\in\Omega_{n,-}^{\rm S},
	\label{1.5}
\end{equation}
where the recursion above and
$c_j[\check q](0)=(-1)^{j-1}c_j[q](0)$ have been used.

Theorem~\ref{ t1} also supplies a spectral hypothesis in
Kotani's construction of KdV flows \cite{Kotani}.  Let $p\geq3$ be odd.
Choose a smooth positive even function $\omega$ such that, for large
$|\eta|$,
\[
\omega(\eta)\geq\gamma_{p-1}(\eta),
\qquad
\omega(\eta)\asymp(1+|\eta|)^{-(p-1)},
\]
and put
\[
D_- =\{k=\lambda+i\eta:|\lambda|>\omega(\eta)\}.
\]
On a compact set the curve may be modified so that its interior contains
$[-\mu_0,\mu_0]$, where $\lambda_0<0$ is a lower bound for the whole-line
Schr\"odinger operator and $\mu_0=\sqrt{-\lambda_0}$.  This compact
modification does not affect the high-energy argument.  In Kotani's sign
convention,
the single function on the two components is
\[
m(k)=
\begin{cases}
	-m_+^{\rm S}(0,-k^2),& \operatorname{Re}k>0,\\
	m_-^{\rm S}(0,-k^2),& \operatorname{Re}k<0.
\end{cases}
\]
If $q\in W^{2p-1,\infty}(\mathbb{R};\mathbb{R})$, then
Theorem~\ref{ t1}, applied to $q|_{\mathbb R_+}$ and to the
reflected potential $\check q$, with $N=2p-1$ and $n=p-1$, gives
\[
m(k)
=k-\sum_{j=1}^{p-1}c_j(0)k^{-j}+O(|k|^{-p})
\]
throughout $D_-$.  This verifies Kotani's condition
(M.2) with $L=p+1$.  Since a bounded real potential is bounded from
below, the Herglotz properties of the two half-line Weyl functions give the
nonreal positivity requirement in (M.1), while the reality of $q$ gives
$m(\overline{k})=\overline{m(k)}$.  To verify the remaining real-axis
positivity explicitly, let $G(z;x,y)$ denote the integral kernel of
$(H_q-z)^{-1}$.  If $x>\mu_0$, then $-x^2<\lambda_0$, and hence
\[
m(x)-m(-x)
=-\bigl(m_+^{\rm S}(0,-x^2)+m_-^{\rm S}(0,-x^2)\bigr)
=\frac{1}{G(-x^2;0,0)}>0.
\]
Here the last inequality follows from positivity of the resolvent below the
spectrum.  Together with the preceding choice of the compact part of the
curve, this verifies condition (M.1).  Kotani's flow theorem therefore
yields the following consequence.

\begin{corollary}
	\label{c2}Let $p\geq3$ be odd.  If
	$q\in W^{2p-1,\infty}(\mathbb{R};\mathbb{R})$, then Kotani's construction
	produces a global classical solution of the member of the KdV hierarchy
	indexed by $(p+1)/2$, with initial value $q$.
\end{corollary}

This improves the regularity threshold obtained in \cite{Kotani} for the
ergodic sufficient condition.  Indeed, Kotani proves that an ergodic
potential in $C_b^m(\mathbb{R})$ belongs almost surely to the relevant
spectral class whenever
$L\leq\bigl(m-3(p-1)\bigr)/6$; choosing the minimal value $L=p+1$
requires $m\geq9p+3$.
Theorem~\ref{ t1} verifies the required
Weyl asymptotic directly, removes the ergodicity assumption, and reduces
the regularity requirement to $2p-1$ bounded derivatives.

Let $\boldsymbol f_+(\cdot,z)$ and $\boldsymbol f_-(\cdot,z)$ denote
solutions of a Dirac equation square integrable near $+\infty$ and
$-\infty$, respectively.  For the displayed choice of the free diagonal
part, the natural small scalar coordinates are
\begin{equation}
	m_+(z)=\frac{f_{+,2}(0,z)}{f_{+,1}(0,z)}
	\quad\text{for }z\in\mathbb{C}_-,
	\qquad
	m_-(z)=\frac{f_{-,2}(0,z)}{f_{-,1}(0,z)}
	\quad\text{for }z\in\mathbb{C}_+.
	\label{1.6}
\end{equation}
Thus the two asymptotically small coordinates are naturally attached to
different half-planes.  In the opposite half-plane the Weyl line remains
well-defined, but the reciprocal quotient is the natural finite
coordinate.  Set
\begin{align*}
	\Omega_{n,+}^{\rm D}
	&=\{z\in\mathbb{C}_+:
	\operatorname{Im}z>\gamma_n(\operatorname{Re}z)\},\\
	\Omega_{n,-}^{\rm D}
	&=\{z\in\mathbb{C}_-:
	-\operatorname{Im}z>\gamma_n(\operatorname{Re}z)\}.
\end{align*}

\begin{theorem}
	\label{t3}
	Let $q\in W^{N,\infty}(\mathbb{R};\mathbb{C})$ and let $m_\pm$ be the
	Weyl functions of $D_q^{\rm sa}$ defined by
	\eqref{1.6}.  Then, as $|z|\to\infty$,
	\begin{equation}
		m_{\pm}(z)=\sum_{j=1}^{N-n-1}c_j(0)z^{-j}
		+O(|z|^{n-N}),\qquad z\in\Omega_{n,\mp}^{\rm D}.
		\label{1.7}
	\end{equation}
	The coefficients are determined by
	\[
	c_1=\frac{\overline q}{2},
	\qquad
	c_{j+1}=-\frac{i}{2}c_j'
	+\frac{q}{2}\sum_{\ell=1}^{j-1}c_\ell c_{j-\ell},
	\qquad j\geq1.
	\]
\end{theorem}

\begin{theorem}
	\label{t4}
	Let $q\in W^{N,\infty}(\mathbb{R};\mathbb{C})$ and let $m_\pm$ be the
	Weyl functions of $D_q^{\rm ns}$ defined by
	\eqref{1.6}.  If the lower-bound constant $c_n$ in
	\eqref{1.2} is chosen sufficiently large, then $m_+$ is analytic in
	$\Omega_{n,-}^{\rm D}$ and $m_-$ is analytic in
	$\Omega_{n,+}^{\rm D}$.  Moreover,
	\begin{equation}
		m_{\pm}(z)=\sum_{j=1}^{N-n-1}c_j(0)z^{-j}+O(|z|^{n-N}),
		\qquad z\in\Omega_{n,\mp}^{\rm D},
		\label{1.8}
	\end{equation}
	The coefficients are determined by
	\[
	c_1=-\frac{\overline q}{2},
	\qquad
	c_{j+1}=-\frac{i}{2}c_j'
	+\frac{q}{2}\sum_{\ell=1}^{j-1}c_\ell c_{j-\ell},
	\qquad j\geq1.
	\]
\end{theorem}

The novelty of Theorems~\ref{ t1},
\ref{t3} and \ref{t4} is
the uniformity of the expansions in domains that close up against the
real axis at the polynomial rate prescribed by \eqref{1.2}.  These
domains contain high-energy sequences for which
$|\operatorname{Im}z|/|z|\to0$ and therefore lie genuinely beyond the
fixed-sector results of Harris.  The loss of $n$ powers in the remainder
is the cost of inverting the weakened exponential decay.

The hypotheses of Theorem~\ref{ t1} and
Corollary~\ref{c2} should be regarded as sufficient conditions,
and no necessity or optimality is claimed here.  In particular, the
boundedness assumptions on the potential and its derivatives provide
a convenient means of obtaining uniform control in the diagonalization
and fixed-point arguments, but they are not expected to be necessary
for near-spectrum Weyl asymptotics.  The harmonic-oscillator example in
Section~\ref{sec:open-problem} shows that a complete expansion may
persist for an unbounded potential on $\Omega_{0,+}^{\rm S}$, while it
also indicates that, when the boundary of the spectral domain is
allowed to approach the spectrum more rapidly, the admissible rate
must reflect the growth of the potential at infinity.  Determining the
optimal balance between regularity, growth at infinity, and the rate
of approach to the spectrum remains an open problem.

Section 2 proves a unified near-identity diagonalization theorem for
$2\times2$ systems.  Section 3 combines it with Banach's fixed-point
theorem and applies this scheme successively to the Schr\"odinger,
self-adjoint Dirac, and non-self-adjoint Dirac operators.  Section 4
concludes with an unbounded example and the open problem suggested by it.

\section{Diagonalization lemma}
\label{sec:unified-diagonalization}

In this section we establish a diagonalization result that will be used in the proofs below.  For a matrix $X=(x_{pq})_{p,q=1}^{2}$, write
\[
X^{\mathrm d}=\operatorname{diag}(x_{11},x_{22}),
\qquad
X^{\mathrm o}=X-X^{\mathrm d}.
\]
Let
\[
\Lambda=\operatorname{diag}(\lambda _1,\lambda _2),
\qquad \lambda _1\ne\lambda _2,
\qquad
\sigma _3=\operatorname{diag}(1,-1).
\]
For $J\in\{I,\sigma _3\}$, define the reflected $J$-adjoint operation by
\begin{equation}
	\mathcal R_J[F](x,\zeta)
	:=J F(x,\overline\zeta)^*J.
	\label{2.1}
\end{equation}

We first record the elementary calculation behind the successive gauge
transformations.  Put $\operatorname{ad}_{H}(X)=[H,X]$.  If $G=e^{H}$, then
the Campbell expansion and the differentiation formula for an exponential give
\begin{align}
	G^{-1}XG
	&=e^{-H}Xe^{H}
	=\sum_{n=0}^{\infty}\frac{(-1)^n}{n!}
	\operatorname{ad}_{H}^{n}(X),
	\label{2.2}\\
	(G^{-1})'G
	&=-G^{-1}G'
	=-\int_{0}^{1}e^{-sH}H'e^{sH}\,ds =-\sum_{n=0}^{\infty}\frac{(-1)^n}{(n+1)!}
	\operatorname{ad}_{H}^{n}(H').
	\label{2.3}
\end{align}
Consequently, under $y=Gu$, the equation $y'=Ay$ becomes $u'=A^{G}u$ with
\begin{equation}
	A^{G}=G^{-1}AG-G^{-1}G'
	=A+[A,H]-H'
	+\frac12[[A,H],H]+\frac12[H,H']+\cdots .
	\label{2.4}
\end{equation}
The following lemma is the common diagonalization statement needed in
Section~3.

\begin{lemma}
	\label{l5}
	Let $M\geq1$, and consider
	\begin{equation}
		y'=A(x,\zeta)y,
		\qquad
		A(x,\zeta)=\zeta\Lambda+
		\sum_{j=0}^{M-1}\zeta^{-j}A_j(x),
		\label{2.5}
	\end{equation}
	for $x \in I \subset \mathbb R$.  Suppose that
	\[
	A_j\in W^{M-j,\infty}(I;\mathbb C^{2\times2}),
	\qquad 0\leq j\leq M-1,
	\]
	and let $\zeta\to\infty$ in a fixed closed sector on which the powers of
	$\zeta^{-1}$ are defined.  Then there is an analytic transformation
	$y=T_M(x,\zeta)u$ such that, uniformly for $x\in I$,
	\begin{equation}
		T_M=I+O(|\zeta|^{-1}),\qquad
		T_M^{-1}=I+O(|\zeta|^{-1}),\qquad
		\det T_M=1,
		\label{2.6}
	\end{equation}
	and
	\begin{equation}
		u'=\bigl(\zeta\Lambda+D_M(x,\zeta)+R_M(x,\zeta)\bigr)u,
		\label{2.7}
	\end{equation}
	where
	\[
	D_M(x,\zeta)=\sum_{j=0}^{M-1}\zeta^{-j}D_j(x)
	\]
	is diagonal and
	\begin{equation}
		\|R_M(\cdot,\zeta)\|_{L^\infty(I)}=O(|\zeta|^{-M}).
		\label{2.8}
	\end{equation}
	If $A_0=0$, then the first two estimates in
	\eqref{2.6} improve to $T_M^{\pm1}=I+O(|\zeta|^{-2})$.
	If $A_0^{\mathrm d}=0$, then $D_0=0$.
	
	Assume in addition that $\Lambda=\pm i\sigma _3$, that the sector is invariant under complex
	conjugation, and that one of the following identities holds:
	\begin{align}
		A(x,\zeta)\sigma _3
		&=-\sigma _3A(x,\overline\zeta)^*,
		\label{2.9}\\
		A(x,\zeta)
		&=-A(x,\overline\zeta)^*.
		\label{2.10}
	\end{align}
	Then the transformation can be chosen to preserve the corresponding
	symmetry.  More precisely, in case \eqref{2.9},
	\[
	T_M(x,\overline\zeta)^*\sigma _3T_M(x,\zeta)=\sigma _3,
	\]
	whereas in case \eqref{2.10},
	\[
	T_M(x,\overline\zeta)^*T_M(x,\zeta)=I.
	\]
	In either case the transformed coefficient in
	\eqref{2.7} has the same symmetry as $A$.
	In particular, for real $\zeta$, the diagonal entries of
	$D_M(x,\zeta)$ are purely imaginary.
\end{lemma}

\begin{proof}
	We first carry out the diagonalization without imposing a symmetry.  Suppose
	that, after $j$ steps, the coefficient matrix has the form
	\[
	A^{(j)}(x,\zeta)
	=\zeta\Lambda+
	\sum_{\ell=0}^{j-1}\zeta^{-\ell}D_\ell(x)
	+\zeta^{-j}B_j(x)+O(|\zeta|^{-j-1}),
	\]
	where $D_0,\ldots,D_{j-1}$ are diagonal.  For $j=0$ this is simply
	\eqref{2.5}, with $B_0=A_0$.
	
	Since $\lambda _1\ne\lambda _2$, the map
	$X^{\mathrm o}\mapsto[\Lambda,X^{\mathrm o}]$ is invertible on the
	off-diagonal matrices.  Hence the homological equation
	\begin{equation}
		[\Lambda,S_{j+1}]=-B_j^{\mathrm o},
		\qquad S_{j+1}^{\mathrm d}=0,
		\label{2.11}
	\end{equation}
	has a unique solution, namely
	\[
	(S_{j+1})_{pq}
	=-\frac{(B_j)_{pq}}{\lambda _p-\lambda _q}
	\quad(p\ne q),
	\qquad (S_{j+1})_{pp}=0.
	\]
	Set
	\[
	H_{j+1}=\zeta^{-j-1}S_{j+1},
	\qquad
	u_j=e^{H_{j+1}}u_{j+1}.
	\]
	The leading commutator in \eqref{2.4} is
	\begin{equation}
		[\zeta\Lambda,H_{j+1}]
		=\zeta^{-j}[\Lambda,S_{j+1}]
		=-\zeta^{-j}B_j^{\mathrm o}.
		\label{2.12}
	\end{equation}
	All commutators in \eqref{2.4} with the lower-order part of $A^{(j)}$, the derivative
	$H_{j+1}'$, and the terms containing at least two copies of $H_{j+1}$
	are $O(|\zeta|^{-j-1})$.  Therefore
	\begin{align*}
		A^{(j+1)}
		&=e^{-H_{j+1}}A^{(j)}e^{H_{j+1}}
		-e^{-H_{j+1}}(e^{H_{j+1}})'\\
		&=\zeta\Lambda+
		\sum_{\ell=0}^{j-1}\zeta^{-\ell}D_\ell
		+\zeta^{-j}\bigl(B_j+[\Lambda,S_{j+1}]\bigr)
		+O(|\zeta|^{-j-1})\\
		&=\zeta\Lambda+
		\sum_{\ell=0}^{j}\zeta^{-\ell}D_\ell
		+O(|\zeta|^{-j-1}),
		\qquad D_j=B_j^{\mathrm d}.
	\end{align*}
	This proves the induction step.
	
	For completeness, the first step can be seen explicitly to one order beyond
	the cancellation.  With $H_1=\zeta^{-1}S_1$ and
	$[\Lambda,S_1]=-A_0^{\mathrm o}$, formula \eqref{2.4} yields
	\begin{equation}
		A^{(1)}
		=\zeta\Lambda+D_0+\zeta^{-1}B_1+O(|\zeta|^{-2}),
		\qquad D_0=A_0^{\mathrm d},
		\label{2.13}
	\end{equation}
	where
	\[
	B_1=A_1-S_1'+[A_0,S_1]
	+\frac12[[\Lambda,S_1],S_1].
	\]
	The next step takes the off-diagonal part of this precise matrix $B_1$ and
	solves $[\Lambda,S_2]=-B_1^{\mathrm o}$.  Thus the above procedure is not
	merely formal: at each order the already determined coefficient is used in
	the next homological equation.
	
	After $M$ steps, put
	\begin{equation}
		T_M=e^{H_1}e^{H_2}\cdots e^{H_M}.
		\label{2.14}
	\end{equation}
	Since $H_j=O(|\zeta|^{-j})$, this product and its inverse satisfy
	\eqref{2.6}.  Each $S_j$ is off diagonal, so
	$\operatorname{tr}H_j=0$ and
	$\det(e^{H_j})=e^{\operatorname{tr}H_j}=1$; hence $\det T_M=1$.
	The assumed Sobolev regularity supplies exactly the derivatives needed when
	$H_j'$ is formed at the $j$th step, and the expansions above, with their
	integral remainders, are uniform on $I$.  This proves
	\eqref{2.7} and
	\eqref{2.8}.  If $A_0=0$, then $S_1=H_1=0$,
	which gives $T_M^{\pm1}=I+O(|\zeta|^{-2})$.  The assertion
	$A_0^{\mathrm d}=0\Rightarrow D_0=0$ follows from
	\eqref{2.13}.
	
	We now give the symmetry calculation in detail.  Since $J^2=I$, the operation
	introduced in \eqref{2.1} satisfies
	\begin{equation}
		\mathcal R_J[FG]=\mathcal R_J[G]\mathcal R_J[F],
		\qquad
		\mathcal R_J([F,G])
		=-[\mathcal R_J(F),\mathcal R_J(G)].
		\label{2.15}
	\end{equation}
	For $J=\sigma _3$, the identity $\mathcal R_J[A]=-A$ is equivalent to
	\eqref{2.9}; for $J=I$, it is exactly
	\eqref{2.10}.  Moreover,
	$\mathcal R_J[\Lambda]=-\Lambda$ for
	$\Lambda=\pm i\sigma _3$ and either choice of $J$.
	
	Assume inductively that $\mathcal R_J[A^{(j)}]=-A^{(j)}$.
	Comparison of the coefficients in its expansion in powers of $\zeta^{-1}$
	gives
	\begin{equation}
		\mathcal R_J[B_j^{\mathrm o}]=-B_j^{\mathrm o}.
		\label{2.16}
	\end{equation}
	Apply $\mathcal R_J$ to the homological equation.  By
	\eqref{2.15},
	\[
	\mathcal R_J([\Lambda,S_{j+1}])
	=-[\mathcal R_J(\Lambda),\mathcal R_J(S_{j+1})]
	=[\Lambda,\mathcal R_J(S_{j+1})],
	\]
	whereas the right-hand side becomes
	$-\mathcal R_J[B_j^{\mathrm o}]=B_j^{\mathrm o}$.  Hence
	\[
	[\Lambda,\mathcal R_J(S_{j+1})]=B_j^{\mathrm o},
	\qquad
	[\Lambda,-\mathcal R_J(S_{j+1})]=-B_j^{\mathrm o}.
	\]
	Both $S_{j+1}$ and $-\mathcal R_J(S_{j+1})$ are off diagonal and solve
	\eqref{2.11}.  By uniqueness,
	\begin{equation}
		\mathcal R_J[S_{j+1}]=-S_{j+1},
		\qquad
		\mathcal R_J[H_{j+1}]=-H_{j+1}.
		\label{2.17}
	\end{equation}
	
	It remains to verify that a gauge step actually preserves the symmetry of the
	differential equation.  Let $G=e^{H_{j+1}}$.  From
	\eqref{2.17},
	\begin{equation}
		\mathcal R_J[G]=e^{\mathcal R_J[H_{j+1}]}
		=e^{-H_{j+1}}=G^{-1}.
		\label{2.18}
	\end{equation}
	Using the reversed product rule in \eqref{2.15}, we obtain
	\begin{align*}
		\mathcal R_J[G^{-1}A^{(j)}G]
		&=G^{-1}(-A^{(j)})G,\\
		\mathcal R_J[G^{-1}G']
		&=(\mathcal R_J[G])'\mathcal R_J[G^{-1}]
		=(G^{-1})'G
		=-G^{-1}G'.
	\end{align*}
	Consequently,
	\begin{align*}
		\mathcal R_J[A^{(j+1)}]
		&=\mathcal R_J\bigl[G^{-1}A^{(j)}G-G^{-1}G'\bigr]\\
		&=-G^{-1}A^{(j)}G+G^{-1}G'
		=-A^{(j+1)}.
	\end{align*}
	This completes the symmetry induction.  Finally,
	\eqref{2.18} and the order reversal under $\mathcal R_J$ give
	\[
	\mathcal R_J[T_M]
	=e^{-H_M}\cdots e^{-H_1}=T_M^{-1},
	\]
	which is precisely the asserted $J$-unitarity of $T_M$.  Since $J$ and every
	$D_j$ are diagonal, the diagonal part of
	$\mathcal R_J[D_M]=-D_M$ is, for real $\zeta$,
	\[
	D_M(x,\zeta)^*=-D_M(x,\zeta).
	\]
	Thus both diagonal entries of $D_M(x,\zeta)$ are purely imaginary, as
	claimed.
\end{proof}

\section{Proof of theorems}
\label{sec:applications}

In this section we combine the diagonalization lemma with a Volterra
equation for the stable graph.

\subsection{The Schr\"odinger operator}
\begin{proof}[Proof of Theorem~\ref{ t1}]
	We prove the assertion for $m_+^{\rm S}$.  The equation
	$-f''+qf=-k^2f$ is equivalent to
	\[
	\boldsymbol y'=
	\left(k\begin{pmatrix}0&1\\1&0\end{pmatrix}
	+k^{-1}\begin{pmatrix}0&0\\q&0\end{pmatrix}\right)
	\boldsymbol y,
	\qquad
	\boldsymbol y=\begin{pmatrix}f\\k^{-1}f'\end{pmatrix}.
	\]
	Set
	\[
	P=\frac1{\sqrt2}\begin{pmatrix}1&1\\-1&1\end{pmatrix},
	\qquad \zeta=ik,
	\qquad \boldsymbol v=P^{-1}\boldsymbol y.
	\]
	Then
	\begin{equation}
		\boldsymbol v'=A_{S}(x,\zeta)\boldsymbol v,
		\qquad
		A_{S}(x,\zeta)=i\zeta\sigma_3+i\zeta^{-1}B_{S}(x),
		\label{3.1}
	\end{equation}
	where
	\[
	B_{S}=\frac q2\begin{pmatrix}-1&-1\\1&1\end{pmatrix}.
	\]
	Since $q$ is real,
	\[
	A_{S}(x,\zeta)\sigma_3
	=-\sigma_3A_{S}(x,\overline\zeta)^*.
	\]
	Apply Lemma \ref{l5} with
	$M=N+1$, $\Lambda=i\sigma_3$, $A_0=0$, and $A_1=iB_{S}$.
	The regularity requirement is exactly $q\in W^{N,\infty}$.  We obtain
	\begin{equation}
		\boldsymbol v=T_{N+1} \boldsymbol u,
		\qquad
		\boldsymbol u'=\bigl(i\zeta\sigma_3+D_{N+1} 
		+R_{N+1} \bigr)\boldsymbol u,
		\label{3.2}
	\end{equation}
	where
	\begin{equation}
		T_{N+1} =I+O(|k|^{-2}),
		\qquad
		\lVert R_{N+1} (\cdot,\zeta)\rVert_\infty
		\leq C_N|k|^{-N-1}.
		\label{3.3}
	\end{equation}
	The improved estimate for $T_{N+1} $ follows from $A_0=0$, which
	makes the first gauge $H_1$ vanish; it also makes the Laurent expansion of
	$D_{N+1} $ start with $\zeta^{-1}$.  The diagonal entries of
	$D_{N+1} (x,\xi)$ are purely imaginary for real $\xi$.
	
	Write $D_{N+1} =\operatorname{diag}(d_1,d_2)$ and
	$R_{N+1} =(r_{ij})$.  The quotient $\omega=u_2/u_1$ satisfies
	\[
	\omega'=\beta\omega+r_{21}-r_{12}\omega^2,
	\qquad
	\beta=-2i\zeta+d_2-d_1+r_{22}-r_{11}.
	\]
	Let $k=\lambda+i\eta\in\Omega_{n,+}^{\rm S}$.  Since
	$\zeta=ik=-\eta+i\lambda$, one has $\operatorname{Im}\zeta=\lambda$.
	The pure-imaginary property of $D_{N+1} $ on the real axis gives
	\[
	|\operatorname{Re}d_j(x,\zeta)|
	\leq C\lambda|\operatorname{Re}\zeta|^{-2}
	\]
	when $|\operatorname{Re}\zeta|$ is large compared with $\lambda$; if
	$\lambda$ is comparable with $|\zeta|$, then $d_j=O(|\zeta|^{-1})$.
	Together with \eqref{3.3}, this yields
	\begin{equation}
		\operatorname{Re}\beta(x,k)\geq\lambda
		\label{3.4}
	\end{equation}
	for all sufficiently large $|k|$.  Moreover, the lower bound in
	\eqref{1.2} implies $\lambda^{-1}\leq C(1+|k|)^n$.  Put
	\[
	\varepsilon_k=
	\max\{\lVert r_{12}(\cdot,k)\rVert_\infty,
	\lVert r_{21}(\cdot,k)\rVert_\infty\},
	\qquad \rho_k=\frac{\varepsilon_k}{\lambda}.
	\]
	Then
	\[
	\rho_k\leq\frac{C_N}{\lambda|k|^{N+1}}
	=O(|k|^{n-N-1})=o(1).
	\]
	Let $R_k=2\rho_k$ and let $B_{R_k}$ be the closed ball of radius $R_k$
	in $L^\infty(\mathbb R_+)$.  On this complete metric space define an operator by
	\begin{equation*}
		(\mathcal T_k\omega)(x)=-\int_x^\infty
		\exp\left(-\int_x^t\beta(s,k)\,ds\right)
		\bigl(r_{21}(t,k)-r_{12}(t,k)\omega(t)^2\bigr)\,dt.
	\end{equation*}
	By \eqref{3.4}, the exponential kernel is bounded by
	$e^{-\lambda(t-x)}$.  Since $\rho_k=o(1)$, we may assume that
	$\rho_k\leq1/4$.  For $\omega,\omega_1,\omega_2\in B_{R_k}$,
	\begin{align*}
		\lVert\mathcal T_k\omega\rVert_\infty
		&\leq\rho_k(1+R_k^2)\leq R_k,\\
		\lVert\mathcal T_k\omega_1-\mathcal T_k\omega_2\rVert_\infty
		&\leq2\rho_kR_k\lVert\omega_1-\omega_2\rVert_\infty
		\leq\frac14\lVert\omega_1-\omega_2\rVert_\infty.
	\end{align*}
	Thus Banach's fixed-point theorem applies exactly on the stated ball and
	produces a unique fixed point.  Differentiating its integral equation shows
	that it solves the Riccati equation above, and
	\begin{equation}
		\lVert\omega(\cdot,k)\rVert_\infty
		\leq R_k\leq\frac{C_N}{\lambda|k|^{N+1}}.
		\label{3.5}
	\end{equation}
	
	Let $u_1(0,k)=1$ and
	\[
	u_1'=(i\zeta+d_1+r_{11}+r_{12}\omega)u_1,
	\qquad u_2=\omega u_1.
	\]
	The real part of the coefficient on the right is at most $-\lambda/2$;
	thus the resulting solution is square integrable at $+\infty$.
	Write
	\[
	P T_{N+1} (x,ik)=
	\begin{pmatrix}a&b\\c&d\end{pmatrix}.
	\]
	Since $\boldsymbol y=P T_{N+1} \boldsymbol u$, $\boldsymbol y=(f,k^{-1}f')^{\mathsf T}$ and $T_{N+1} =I+O(|k|^{-2})$, it holds that
	\begin{equation}
		m_+^{\rm S}(x,-k^2)
		=k\frac{c+d\omega}
		{a+b\omega}.
		\label{3.6}
	\end{equation}
	Set $p_N =k c/a$.  Since
	$T_{N+1} =I+O(|k|^{-2})$, its Laurent expansion has the form
	\[
	p_N (x,k)=-k+\sum_{j=1}^{N}c_j(x)k^{-j}
	+O(|k|^{-N-1}).
	\]
	The successive off-diagonal cancellations in
	Lemma \ref{l5} are equivalent to comparison of
	powers in the scalar Riccati equation
	$(p_N )'+(p_N )^2=k^2+q$ up to the retained order.  Hence
	\[
	c_1=-\frac q2,
	\qquad
	c_{j+1}=\frac12\left(c_j'
	+\sum_{\ell=1}^{j-1}c_\ell c_{j-\ell}\right).
	\]
	Finally, \eqref{3.6},
	$\det(P T_{N+1} )=1$, and \eqref{3.5} give
	\[
	m_+^{\rm S}(0,-k^2)-p_N (0,k)
	=k\frac{\omega(0,k)}
	{a(0,k)(a(0,k)+b(0,k)\omega(0,k))}
	=O(|k|^{n-N}).
	\]
	Truncating the Laurent expansion after $j=N-n-1$ proves the asserted
	expansion for $m_+^{\rm S}$ and completes the proof of
	Theorem~\ref{ t1}. \bigskip
\end{proof}

\subsection{The self-adjoint Dirac operator}
We write both Dirac equations in the form
\begin{equation}
	\boldsymbol f'(x,z)
	=\bigl(-iz\sigma_3+B(x)\bigr)\boldsymbol f(x,z),
	\qquad
	\sigma_3=\begin{pmatrix}1&0\\0&-1\end{pmatrix}.
	\label{3.7}
\end{equation}
For $D_q^{\rm sa}$ and $D_q^{\rm ns}$, respectively,
\begin{equation*}
	B_{\rm sa}=\begin{pmatrix}0&iq\\-i\overline q&0\end{pmatrix},
	\qquad
	B_{\rm ns}=\begin{pmatrix}0&iq\\ i\overline q&0\end{pmatrix}.
\end{equation*}
With $A_{\rm sa}=-iz\sigma_3+B_{\rm sa}$ and
$A_{\rm ns}=-iz\sigma_3+B_{\rm ns}$, one has
\begin{align}
	A_{\rm sa}(x,z)\sigma_3
	&=-\sigma_3A_{\rm sa}(x,\overline z)^*,
	\notag\\
	A_{\rm ns}(x,z)&=-A_{\rm ns}(x,\overline z)^*.
	\label{3.8}
\end{align}
Apply Lemma \ref{l5} with
$M=N$, $\Lambda=-i\sigma_3$, and $A_0=B_{\rm sa}$ or $B_{\rm ns}$.
In either case $\boldsymbol f=T_N\boldsymbol u$ gives
\begin{equation}
	\boldsymbol u'
	=\bigl(-iz\sigma_3+D_N(x,z)+R_N(x,z)\bigr)\boldsymbol u,
	\label{3.9}
\end{equation}
where $D_N$ is diagonal, its diagonal entries are purely imaginary for
real $z$, and
\begin{equation}
	T_N=I+O(|z|^{-1}),
	\qquad
	\lVert R_N(\cdot,z)\rVert_\infty\leq C_N|z|^{-N}.
	\label{3.10}
\end{equation}
Because $B_{\rm sa}^{\rm d}=B_{\rm ns}^{\rm d}=0$, the expansion
of $D_N$ starts with the power $z^{-1}$.

Write
\begin{equation*}
	D_N=\begin{pmatrix}d_1&0\\0&d_2\end{pmatrix},
	\qquad
	R_N=\begin{pmatrix}r_{11}&r_{12}\\r_{21}&r_{22}\end{pmatrix}.
\end{equation*}
Then \eqref{3.9} becomes
\begin{align*}
	u_1'&=(-iz+d_1+r_{11})u_1+r_{12}u_2,
	\\
	u_2'&=r_{21}u_1+(iz+d_2+r_{22})u_2.
\end{align*}
The quotient $\omega=u_2/u_1$ satisfies
\begin{equation}
	\omega'=\beta\omega+r_{21}-r_{12}\omega^2,
	\qquad
	\beta=2iz+d_2-d_1+r_{22}-r_{11}.
	\label{3.11}
\end{equation}

\begin{proof}[Proof of Theorem~\ref{t3}]
	We prove the assertion for $m_+$ in $\Omega_{n,-}^{\rm D}$; the proof
	for $m_-$ in $\Omega_{n,+}^{\rm D}$ is obtained by integrating from
	$-\infty$ instead of $+\infty$.
	
	Let $z=\lambda-i\eta$, with $\eta>0$.  Since $d_1$ and $d_2$ are
	finite sums of negative powers of $z$ and are purely imaginary for real
	$z$, one has
	\begin{equation}
		|\operatorname{Re}d_j(x,z)|
		\leq C\eta|\lambda|^{-2}
		\label{3.12}
	\end{equation}
	when $|\lambda|$ is large compared with $\eta$.  If $\eta$ is comparable
	with $|z|$, then $d_j=O(|z|^{-1})$.  The lower bound in
	\eqref{1.2} implies
	\begin{equation}
		\eta^{-1}\leq C(1+|z|)^n,
		\qquad z\in\Omega_{n,-}^{\rm D}.
		\label{3.13}
	\end{equation}
	In particular, $|z|^{-N}/\eta=O(|z|^{n-N})=o(1)$.  Since
	$\operatorname{Re}(2iz)=2\eta$, the preceding estimates and
	\eqref{3.10} show that, for $|z|$ sufficiently
	large,
	\begin{equation*}
		\operatorname{Re}\beta(x,z)\geq\eta.
	\end{equation*}
	For the off-diagonal entries of
	\eqref{3.9}, put
	\[
	\varepsilon_z=\max\{\lVert r_{12}(\cdot,z)\rVert_\infty,
	\lVert r_{21}(\cdot,z)\rVert_\infty\}
	\leq C_N|z|^{-N}.
	\]
	Thus its smallness ratio is
	\[
	\rho_z=\frac{\varepsilon_z}{\eta}
	=O(|z|^{n-N})=o(1).
	\]
	Let $R_z=2\rho_z$ and let $B_{R_z}$ be the closed ball of radius $R_z$ in
	$L^\infty(\mathbb R_+)$.  Define $\mathcal T_z$ on $L^\infty(\mathbb R_+)$ by
	\begin{equation}
		(\mathcal T_z\omega)(x,z)
		=-\int_x^\infty
		G(x,t;z)\bigl(r_{21}(t,z)-r_{12}(t,z)\omega(t,z)^2\bigr)\,dt,
		\label{3.14}
	\end{equation}
	where
	\begin{equation*}
		G(x,t;z)
		=\exp\left(-\int_x^t\beta(s,z)\,ds\right),
		\qquad |G(x,t;z)|\leq e^{-\eta(t-x)}.
	\end{equation*}
	The ball $B_{R_z}$ is complete.  Since $\rho_z=o(1)$, for sufficiently
	large $|z|$ we have $\rho_z\leq1/4$, and for
	$\omega,\omega_1,\omega_2\in B_{R_z}$,
	\begin{align*}
		\lVert\mathcal T_z\omega\rVert_\infty
		&\leq\rho_z(1+R_z^2)\leq R_z,\\
		\lVert\mathcal T_z\omega_1-\mathcal T_z\omega_2\rVert_\infty
		&\leq2\rho_zR_z\lVert\omega_1-\omega_2\rVert_\infty
		\leq\frac14\lVert\omega_1-\omega_2\rVert_\infty.
	\end{align*}
	Consequently, Banach's fixed-point theorem gives a unique fixed point in
	$B_{R_z}$, which satisfies
	\begin{equation}
		\left\lVert\omega(\cdot,z)\right\rVert_\infty
		\leq\frac{2C_N}{\eta|z|^N}
		=O(|z|^{n-N}).
		\label{3.15}
	\end{equation}
	For this fixed point, differentiating the identity
	$\omega=\mathcal T_z\omega$, with $\mathcal T_z$ defined in
	\eqref{3.14}, gives
	\eqref{3.11}.  To justify analytic dependence, fix
	a compact parameter set $K$ in the high-energy region, put
	$\rho_K=\sup_{z\in K}\rho_z$ and use the common ball of radius
	$2\rho_K$.  The two estimates above hold there uniformly.  Starting the
	Picard iteration at zero, the analytic iterates therefore converge locally
	uniformly with values in $L^\infty$, and $\omega$ is analytic in $z$.
	
	To verify the Weyl condition, let $u_1$ solve
	\begin{equation*}
		u_1'=(-iz+d_1+r_{11}+r_{12}\omega)u_1,
		\qquad u_1(0,z)=1.
	\end{equation*}
	Then
	\begin{equation*}
		\boldsymbol u(x,z)
		=u_1(x,z)\begin{pmatrix}1\\\omega(x,z)\end{pmatrix}
	\end{equation*}
	solves \eqref{3.9}.  For $|z|$ sufficiently
	large,
	\begin{equation*}
		\operatorname{Re}(-iz+d_1+r_{11}+r_{12}\omega)
		\leq-\frac{\eta}{2}.
	\end{equation*}
	Thus $\boldsymbol u\in L^2(\mathbb{R}_+)^2$, and so is
	$\boldsymbol f=T_N\boldsymbol u$.
	
	Write
	\begin{equation*}
		T_N(0,z)=\begin{pmatrix}a(z)&b(z)\\c(z)&d(z)\end{pmatrix}.
	\end{equation*}
	Since $T_N=I+O(|z|^{-1})$, the first component of
	$T_N(0,z)(1,\omega(0,z))^{\mathsf T}$ does not vanish for large $|z|$,
	and
	\begin{equation}
		m_+(z)=\frac{c+d\omega}{a+b\omega}.
		\label{3.16}
	\end{equation}
	Set $p_N=c/a$.  The construction of $T_N$ gives
	\begin{equation*}
		p_N(z)=\sum_{j=1}^{N-n-1}c_j(0)z^{-j}
		+O(|z|^{n-N}).
	\end{equation*}
	Moreover,
	\begin{equation}
		m_+(z)-p_N(z)
		=\frac{\omega(ad-bc)}{a(a+b\omega)}
		=O(|z|^{n-N}).
		\label{3.17}
	\end{equation}
	For the left endpoint, let
	$z=\lambda+i\eta\in\Omega_{n,+}^{\rm D}$.  Then
	$\operatorname{Re}\beta\leq-\eta$.  Applying the same contraction
	estimates to the reversed Volterra operator
	\[
	(\mathcal T_z^-\omega)(x)=\int_{-\infty}^x
	\exp\left(\int_t^x\beta(s,z)\,ds\right)
	\bigl(r_{21}(t,z)-r_{12}(t,z)\omega(t,z)^2\bigr)\,dt,
	\]
	whose kernel is bounded by $e^{-\eta(x-t)}$, constructs the left Weyl
	solution.  The same calculation gives the
	asserted expansion for $m_-$ in \eqref{1.7}.  This also
	explains the pairing $m_+\leftrightarrow\Omega_{n,-}^{\rm D}$ and
	$m_-\leftrightarrow\Omega_{n,+}^{\rm D}$ in the theorem statement.
	
	It remains to identify the coefficients.  For self-adjoint Dirac operators, the
	Weyl function $m_{+}$ satisfies a scalar Riccati equation
	\[
	\partial_x m_{+}(x,z)=-i\overline q(x)+2izm_{+}(x,z)-iqm_{+}(x,z)^2.
	\]
	Comparison of the coefficients of $z^{-j}$ yields
	\[
	c_1=\frac{\overline q}{2},
	\qquad
	c_{j+1}=-\frac{i}{2}c_j'
	+\frac{q}{2}\sum_{\ell=1}^{j-1}c_\ell c_{j-\ell},
	\qquad j\geq1.
	\]
	This completes the proof of Theorem~\ref{t3}. \bigskip
\end{proof}

\subsection{The non-self-adjoint Dirac operator}
\begin{proof}[Proof of Theorem~\ref{t4}]
	For $B=B_{\rm ns}$, the coefficient matrix in
	\eqref{3.7} is skew-Hermitian when $z$ is real;
	see \eqref{3.8}.  Hence
	Lemma \ref{l5} gives a diagonal matrix
	$D_N(x,\lambda)$ with purely imaginary
	entries for real $\lambda$.  The gap estimate
	\eqref{3.13}, the contraction argument in the proof of
	Theorem~\ref{t3}, and the $L^2$ reconstruction are therefore
	unchanged.  They construct a right
	$L^2$ line in $\Omega_{n,-}^{\rm D}$ and a left $L^2$ line in
	$\Omega_{n,+}^{\rm D}$ for $c_{n}$ sufficiently large, and the
	locally uniform Picard iteration used there makes these lines analytic.
	
	For completeness, the fixed-point theorem by itself gives uniqueness only
	inside its small $L^\infty$ ball.  Uniqueness of the Weyl line follows
	separately from the zero-trace structure.  If $\boldsymbol f$ and
	$\boldsymbol g$ are two $L^2$ solutions at the same endpoint, then
	\[
	W(x)=\det\bigl(\boldsymbol f(x),\boldsymbol g(x)\bigr)
	\]
	is constant because the coefficient matrix in
	\eqref{3.7} has trace zero.  Along a sequence
	approaching that endpoint both vectors tend to zero, so $W=0$ and the two
	solutions are linearly dependent.  Thus the constructed line is the unique
	Weyl line.
	
	Finally, \eqref{3.16}--\eqref{3.17}
	give the expansions in \eqref{1.8}.  In this case the
	Riccati equation is
	\[
	\partial_xm=i\overline q+2izm-iqm^2.
	\]
	Comparison of powers of $z^{-1}$ gives
	\[
	c_1=-\frac{\overline q}{2},
	\qquad
	c_{j+1}=-\frac{i}{2}c_j'
	+\frac{q}{2}\sum_{\ell=1}^{j-1}c_\ell c_{j-\ell},
	\qquad j\geq1.
	\]
	This proves Theorem~\ref{t4}. \bigskip
\end{proof}

\section{An unbounded example and an open problem}
\label{sec:open-problem}

The boundedness assumption in Theorem~\ref{ t1} is a convenient
sufficient condition, but it is not expected to be necessary.  A useful
test case is the harmonic oscillator
\[
q(x)=cx^2,\qquad c>0,\qquad x\in\mathbb R_+.
\]
The corresponding Weyl function follows directly from the standard
parabolic-cylinder representation.  Indeed, set
\[
t=\sqrt{2}\,c^{1/4}x,
\qquad
\nu=\frac{z}{2\sqrt c}-\frac12,
\]
and write $\psi(x)=u(t)$.  Then $-\psi''+cx^2\psi=z\psi$ becomes
\[
u''(t)+\left(\nu+\frac12-\frac{t^2}{4}\right)u(t)=0.
\]
The solution square integrable at $+\infty$ is, up to a nonzero factor,
\[
\psi_+(x,z)=D_\nu\!\left(\sqrt{2}\,c^{1/4}x\right),
\]
where $D_\nu$ is the parabolic cylinder function; its standard large-$t$
asymptotic is $D_\nu(t)\sim t^\nu e^{-t^2/4}$ on the positive real axis
\cite{NIST}.  The values at the origin are
\[
D_\nu(0)=
\frac{2^{\nu/2}\sqrt\pi}{\Gamma((1-\nu)/2)},
\qquad
D_\nu'(0)=
-\frac{2^{(\nu+1)/2}\sqrt\pi}{\Gamma(-\nu/2)}.
\]
Consequently,
\[
m_+^{\rm S}(0,z)
=\sqrt{2}\,c^{1/4}\frac{D_\nu'(0)}{D_\nu(0)},
\]
and hence
\begin{align}
	m_+^{\rm S}(0,z)
	&=-2c^{1/4}
	\frac{\Gamma\bigl((3-c^{-1/2}z)/4\bigr)}
	{\Gamma\bigl((1-c^{-1/2}z)/4\bigr)} \notag\\
	&=2c^{1/4}
	\tan\!\left(\frac{\pi(c^{-1/2}z-1)}4\right)
	\frac{\Gamma\bigl((3+c^{-1/2}z)/4\bigr)}
	{\Gamma\bigl((1+c^{-1/2}z)/4\bigr)}.
	\label{4.1}
\end{align}
The second identity follows from the reflection formula for the Gamma
function.  If $w=(1+c^{-1/2}z)/4$, then, uniformly in every closed sector
$|\arg w|\leq\pi-\delta$,
\begin{equation}
	\frac{\Gamma(w+1/2)}{\Gamma(w)}
	\sim w^{1/2}\left(1+a_1w^{-1}+a_2w^{-2}+\cdots\right).
	\label{4.2}
\end{equation}
Thus the behavior near the positive real axis is governed by the tangent
factor in \eqref{4.1}.

To see the dependence on the rate of approach, fix $a>C_n$ and take
\[
k=a(1+|\eta|)^{-n}+i\eta,
\qquad z=-k^2.
\]
By \eqref{1.2}, this path lies in $\Omega_{n,+}^{\rm S}$ for all
sufficiently large $|\eta|$.  For definiteness, let $\eta\to+\infty$.
If $\theta(z)=\pi(c^{-1/2}z-1)/4$, then
\begin{equation}
	\operatorname{Im}\theta(-k^2)
	=-\frac{\pi a}{2\sqrt c}\,\eta(1+\eta)^{-n}.
	\label{4.3}
\end{equation}
For $n=0$, this imaginary part tends to $-\infty$, and
$\tan\theta(-k^2)=-i+O(e^{-C|k|})$ along the boundary-scale path.
In fact, the expansion is uniform throughout $\Omega_{0,+}^{\rm S}$.
Indeed, write $k=\lambda+i\eta$.  Since \eqref{1.2} gives
$\lambda\geq c_0>0$, split the region into $|\eta|\leq2\lambda$ and
$|\eta|>2\lambda$.  In the first part,
\[
W=\frac{1+c^{-1/2}k^2}{4}
\]
stays in a fixed closed sector avoiding the negative real axis, so the
first Gamma quotient in \eqref{4.1} has the usual
uniform ratio expansion.  In the second part,
\[
w=\frac{1-c^{-1/2}k^2}{4}
\]
lies in the right half-plane, while
$|\operatorname{Im}\theta(-k^2)|\geq C|\eta|$.  Hence
\[
\tan\theta(-k^2)
=-i\,\operatorname{sgn}\eta+O(e^{-C|\eta|}),
\]
and the second Gamma quotient can be treated by
\eqref{4.2}.  The two estimates give a uniform
complete inverse-power expansion on $\Omega_{0,+}^{\rm S}$, with
coefficients agreeing with the Riccati recursion in
Theorem~\ref{ t1}.  Thus the harmonic oscillator is an unbounded
potential for which the full algebraic expansion survives when $n=0$.

For $n=1$, the tangent factor remains bounded along the displayed path but
is generally oscillatory, so this is a borderline regime: boundedness alone
does not yield a coefficientwise expansion.  For $n\geq2$, the imaginary
part in \eqref{4.3} tends to zero, and the
tangent factor becomes arbitrarily large along sequences approaching its
poles.  Consequently, no uniform expansion of the form in
Theorem~\ref{ t1} can hold on $\Omega_{n,+}^{\rm S}$ in general.

This example suggests that the optimal hypotheses should involve a
quantitative balance between two features of the potential: its degree of
smoothness, which determines the length of the inverse-power expansion,
and its growth at infinity, which determines how closely the spectral
parameter may approach the spectrum while retaining uniform control.  A
natural open problem is to formulate such a balance and to characterize it
directly in terms of the Titchmarsh--Weyl function.  In the KdV setting,
this would complement Kotani's construction by allowing one to infer
spatial regularity or growth properties of the evolved potential from the
near-spectrum class of its spectral data, and may therefore help in the
study of global KdV behavior as $t\to\pm\infty$.

\end{document}